\documentclass[10pt,reqno]{amsart}
\usepackage{amsmath}
\usepackage{amsthm}
\usepackage{amssymb}
\usepackage{amsfonts}
\usepackage{latexsym}
\usepackage{hyperref}
\usepackage{cite}
\usepackage{xcolor}
\usepackage{wasysym}

\usepackage{mathtools} 
\mathtoolsset{showonlyrefs} 

\theoremstyle{plain}
\newtheorem{theorem}{Theorem}

\newtheorem{lemma}[theorem]{Lemma}

\theoremstyle{definition}

\newtheorem{remark}[theorem]{Remark}

\newcommand{\Q}{\mathbb{Q}}

\newcommand{\C}{\mathbb{C}}
\newcommand{\N}{\mathbb{N}}
\newcommand{\F}{\mathbb{F}}
\newcommand{\K}{\mathbb{K}}

\newcommand{\Z}{\mathbb{Z}}
\newcommand{\PP}{\mathfrak{p}}
\renewcommand{\a}{\mathfrak{a}}
\renewcommand{\O}{\mathcal{O}}

\renewcommand{\AA}{\mathbf{A}}
\newcommand{\BB}{\mathbf{B}}
\newcommand{\CC}{\mathbf{C}}
\renewcommand{\P}{\mathbf{P}}

\newcommand{\D}{\mathbf{D}}

\newcommand{\M}{\mathbf{M}}

\renewcommand{\Re}{\operatorname{Re}}

\renewcommand{\pmod}[1]{\,\,(\operatorname{mod}#1)}

\let\oldenumerate=\enumerate
	\def\enumerate{
	\oldenumerate
	\setlength{\itemsep}{5pt}
	}
\let\olditemize=\itemize
	\def\itemize{
	\olditemize
	\setlength{\itemsep}{5pt}
	}

\allowdisplaybreaks

\begin{document}

\title[An effective analytic formula ]{An effective analytic formula for the number of distinct irreducible factors of a polynomial}

	\author[S.~R.~Garcia]{Stephan Ramon Garcia}
	\address{Department of Mathematics, Pomona College, 610 N. College Ave., Claremont, CA 91711} 
	\email{stephan.garcia@pomona.edu}
	\urladdr{\url{http://pages.pomona.edu/~sg064747}}
	
	\author[E.~S.~Lee]{Ethan Simpson Lee}
    \address{School of Science, UNSW Canberra at the Australian Defence Force Academy, Northcott Drive, Campbell, ACT 2612} 
    \email{ethan.s.lee@student.adfa.edu.au}
    \urladdr{\url{https://www.unsw.adfa.edu.au/our-people/mr-ethan-lee}}

	\author[J.~Suh]{Josh Suh}
	\author[J.~Yu]{Jiahui Yu}
	
\thanks{SRG supported by NSF Grant DMS-1800123.}



\begin{abstract}
We obtain an effective analytic formula, with explicit constants, for the number of distinct irreducible factors of a polynomial $f \in \Z[x]$.  We use an explicit version of Mertens' theorem for number fields to estimate a  related sum over rational primes.  For a given $f \in \Z[x]$, our result yields a finite list of primes that certifies the number of distinct irreducible factors of $f$.
 \end{abstract}

\maketitle

\section{Introduction}
In this note we establish an effective analytic formula (Theorem \ref{Theorem:Main}) 
for the number of distinct irreducible factors of a polynomial $f \in \Z[x]$.  
Our error bounds are unconditional and explicit in terms of 
their dependence upon $f$.   
To this end, we first introduce Theorem \ref{Theorem:NFM}, which is of independent interest 
since it relates a Mertens-type sum over number fields to a weighted sum over rational primes
in an explicit manner that does not involve the residue of a Dedekind zeta function.

\subsection{A Mertens-type sum}

Let $\K$ be a number field of degree $d$ with ring of integers of $\O_\K$.
Let $N(\PP)$ denote the norm of a prime ideal $\PP\subset\O_{\K}$ and $p$ a rational prime.
The primitive element theorem says that $\K = \Q(\alpha)$ for some $\alpha \in \O_{\K}$ \cite[Cor.~2.12]{StewartTall}.
Let $g \in \Z[x]$ be irreducible with root $\alpha$ and leading coefficient $c$.
The degree of $g$ is $d$ \cite[p.~47]{StewartTall} and the discriminant $D_g$ of $g$ is nonzero 
(Lemma \ref{Lemma:IrreducibleDiscriminant}).

\begin{theorem}\label{Theorem:NFM}
Let $\K = \Q(\alpha)$, in which $g \in \Z[x]$ is irreducible with root $\alpha$, leading coefficient $c$, and degree $d$.  Define $\D_g = |c|^{ (d-1)(d-2) } |D_g|$. 
For $x > \max\{2,\sqrt{\D_g}\}$,
\begin{equation*}
    \sum_{p\leq x} \frac{\omega_g(p)}{p} = \underbrace{\sum_{N(\PP)\leq x} \frac{1}{N(\PP)} }_{\M_{\K}(x)}+ A_g,
\end{equation*}
in which $\omega_g(p)$ is the number of solutions to  $g(x) \equiv 0 \pmod{p}$,
\begin{equation}\label{eq:Ah}
    |A_g| \leq d\left( \M_{\Q}(|c|) + \M_{\Q}\Big(\sqrt{\D_g}\Big) + 0.64\right),\quad\text{and}\quad \M_{\Q}(x) = \sum_{p\leq x}\frac{1}{p}.
\end{equation}
\end{theorem}

Rosser--Schoenfeld bounded $\M_{\Q}(x)$ explicitly \cite[(3.20)]{Rosser};
see \eqref{eq:ClassicalMertens} below.  It is known that $\M_{\K}(x) = \log \log x + O(1)$, where the
$O(1)$ term depends upon the residue of the corresponding Dedekind zeta function at $s=1$
\cite{Rosen, EMT4NF2}. 
Theorem \ref{Theorem:NFM} avoids this inconvenience and reduces the computation 
of a Mertens-type sum over number fields to a sum over rational primes, with an explicit error bound.

\subsection{An effective analytic formula}
Recall that $f \in \Z[x]$ is \emph{primitive} if the greatest common divisor
of its coefficients is $1$. A nonconstant polynomial in $\Z[x]$ is \textit{irreducible} in $\Z[x]$ if and only if
it is primitive and irreducible in $\Q[x]$. Gauss' primitivity lemma ensures that the product of primitive polynomials is primitive, so we may assume that each irreducible factor of a given $f\in \Z[x]$ is primitive.

Suppose that $f = f_1 f_2 \cdots f_k \in \Z[x]$, in which $f_1,f_2,\ldots,f_k \in \Z[x]$ are irreducible;
they are uniquely determined up to ordering.
Since the Euclidean algorithm reveals any common factors of $g,g'\in\Z[x]$, we may assume that the $f_i$ are distinct without loss of generality. Equivalently, the discriminant $D_f$ of $f$ is nonzero.
Under these circumstances, we provide an analytic formula for $k$.
Our result is unconditional and explicit.  The error term depends only upon the degree and discriminant of $f$.

\begin{theorem}\label{Theorem:Main}
Suppose that $f =f_1f_2,\cdots f_k \in \Z[x]$ is a product of distinct, irreducible nonconstant polynomials $f_1,f_2,\ldots,f_k \in \Z[x]$.
Let $f$ have degree $d \geq 1$ and leading coefficient $c$.
Write $\D_f = |c|^{(d-1)(d-2)} |D_f|$, in which $D_f$ is the discriminant of $f$.  
For $x\geq \max\{2, |D_f|, \sqrt{\D_f}\}$,
\begin{equation}\label{eq:MainUnconditional}
    \left| \frac{1}{\log\log x}\sum_{p \leq x} \frac{\omega_f(p)}{p} - k \right|
    \, \leq \, \frac{d\,\M_{\Q}( |D_f| ) + \AA + \BB(x) + \CC}{\log\log{x}},
\end{equation}
in which
\begin{align*}
    \AA &\leq d\big( \M_{\Q}(|c|) + \M_{\Q}( \sqrt{\D_f}) + 0.64 \big),\\[3pt]
    \BB(x) &\leq \frac{2 }{\log x} \left( 
 \left[ \frac{\Lambda \sqrt{\D_f}}{0.36232}\big(  0.55 d^2  + 44.86 d \big) \right] + 2 d \right)    
    , \quad \text{and}\\[3pt]
    \CC &\leq d \Big(\gamma + 1.02\,d - 0.02 + \frac{d-1}{2}\log \D_f\Big).
\end{align*}
Here $\gamma =0.57721\ldots$ is the Euler--Mascheroni constant, $\M_{\Q}$ is given by \eqref{eq:Ah}, and
\begin{equation*}
\Lambda = e^{28.2 d+5}(d+1)^{\frac{5d+5}{2}}|D_f|(\log |D_f| )^d.
\end{equation*}
\end{theorem}

Theorem \ref{Theorem:Main} produces a finite list of primes that certifies 
$f \in \Z[x]$ has exactly $k$ distinct irreducible factors: 
take $x\geq \max\{2, |D_f|, \sqrt{\D_f}\}$ such that the right-hand side 
of \eqref{eq:MainUnconditional} is less than $0.5$.
Although this is not yet practical, Theorem \ref{Theorem:Main} is a valuable proof of concept, and several avenues for improvement are discussed in Section \ref{Section:Conclusion}.
Table \ref{Table:Only} exhibits the behavior of the main term in Theorem \ref{Theorem:Main}.

The proof of Theorem \ref{Theorem:Main} relies upon Theorem \ref{Theorem:NFM} and its proof
requires recent work on explicit Mertens' theorems for number fields and residue bounds for Dedekind zeta functions \cite{EMT4NF2}. We must consider the fields $\K_i$ generated by the roots of $f_i$ and quantify the resulting error in terms of the degree, discriminant and leading coefficient of $f$.
Moreover, each step must be uniform and explicit.

\begin{table}
\begin{equation*}
\begin{array}{c|ccccc}
f(x) & \text{Factorization} & k & F(100) & F(1{,}000) & F(10{,}000) \\
\hline
x^4+1 & x^4+1& 1 & 0.6377 & 0.6729 & 0.7108\\
x^4 + 4 & (x^2-2)(x^2+2) & 2 & 1.6164 & 1.6712 & 1.7109\\
x^4-1 &(x^2-1)(x-1)(x+1)& 3 & 2.6781 & 2.7222 & 2.7543\\
x^4-5x^2 + 4 & (x-1)(x+1)(x-2)(x+2) & 4 & 3.6306 & 3.6870 & 3.7227\\
\end{array}
\end{equation*}
\caption{$F(x) = \tfrac{1}{\log\log{x}}\sum_{p\leq x} \frac{\omega_f(p)}{p}$, the main term in Theorem \ref{Theorem:Main}.}
\label{Table:Only}
\end{table}

\subsection*{Structure}
This paper is structured as follows.
The preliminaries are covered in Section \ref{Section:Preliminaries}.
Section \ref{Section:Mertens} contains the proof of Theorem \ref{Theorem:NFM}.
The proof of Theorem \ref{Theorem:Main} is the focus of Section \ref{Section:Main}.
Finally, we consider future avenues of research in Section \ref{Section:Conclusion}.

\subsection*{Acknowledgments}
The authors thank Lenny Fukshansky, Florian Luca, Valeriia Starichkova, and
Tim Trudgian for helpful comments and suggestions.

\section{Preliminaries}\label{Section:Preliminaries}
In this section, we review results needed for the proof of Theorem \ref{Theorem:Main}.
Subsection \ref{Subsection:Discriminant} concerns several important notions from elimination theory,
while Subsection \ref{Subsection:Dedekind} presents explicit unconditional bounds on Dedekind zeta residues.
We explore a Mertens-type sums for algebraic number fields in Subsection \ref{Subsection:Mertens}.
Finally, Subsection \ref{Subsection:PrimeZeta} contains a few remarks about a function related to the prime zeta function.

\subsection{Resultants and discriminants}\label{Subsection:Discriminant}
Although the material in this subsection is classical, it is surprisingly difficult to
find it all stated in one convenient reference.
Let $f(x) = a_n x^n + a_{n-1} x^{n-1} + \cdots + a_0$
and $g(x) = b_m x^m + b_{m-1} x^{m-1} + \cdots + b_0$ be complex polynomials of positive degrees $n$ and $m$, respectively.
Their \emph{resultant} $R(f(x),g(x);x)$ is the determinant of the 
$(m+n)\times(m+n)$ \emph{Sylvester matrix}
\begin{equation}\label{eq:Sylvester}
{\footnotesize
\begin{bmatrix}
a_{n}&0&\cdots &0&b_{m}&0&\cdots &0\\
a_{n-1}&a_{n}&\cdots &0&b_{m-1}&b_{m}&\cdots &0\\[-3pt]
a_{n-2}&a_{n-1}&\ddots &0&b_{m-2}&b_{m-1}&\ddots &0\\
\vdots &\vdots &\ddots &a_{n}&\vdots &\vdots &\ddots &b_{m}\\[-3pt]
a_{0}&a_{1}&\cdots &\vdots &b_{0}&b_{1}&\cdots &\vdots \\[-3pt]
0&a_{0}&\ddots &\vdots &0&b_{0}&\ddots &\vdots \\
\vdots &\vdots &\ddots &a_{1}&\vdots &\vdots &\ddots &b_{1}\\
0&0&\cdots &a_{0}&0&0&\cdots &b_{0}
\end{bmatrix}.
}
\end{equation}
We often suppress the variable $x$ and write $R(f,g)$ instead of
$R(f(x),g(x);x)$.   Since \eqref{eq:Sylvester} has $m$ columns with the coefficients of $f$
and $n$ columns with the coefficients of $g$, it follows that
$R( \alpha  f, \beta g) = \alpha^m \beta^n R(f,g)$ for $\alpha , \beta \in \C$.
The definition ensures that $R(f,g) \in \Z$ whenever $f,g \in \Z[x]$.

We now need an important fact about polynomials.
Suppose $f\in \Z[x]$ and $f = AB$, in which $A,B\in \Q[x]$. 
Gauss' lemma ensures that $f = ab$, where $a,b \in \Z[x]$ are of the form
$a = \mu A$ and $b = \mu^{-1} B$ for some $\mu \in \Q$ \cite[Prop.~5, p.~303]{DF}.

\begin{lemma}\label{Lemma:ZeroResult}
Let $f,g \in \Z[x]$ have positive degree.  Then $R(f,g) = 0$
if and only if $f$ and $g$ have a common divisor in $\Z[x]$ of positive degree.
\end{lemma}

\begin{proof}
Let $\mathcal{P}_i$ denote the $\C$-vector space
of polynomials of degree at most $i-1$.  Then the matrix in \eqref{eq:Sylvester} 
represents the linear map $(u,v) \mapsto uf+vg$ from $\mathcal{P}_{m} \times \mathcal{P}_{n}$
to $\mathcal{P}_{m+n}$ with respect to the corresponding monomial bases (listed in order of descending degree).
Thus, $R(f,g) = 0$ if and only if $f,g$ have a common divisor in $\C[x]$ of positive degree, that is,
if and only if $f,g$ share a common root in $\alpha \in \C$.
Since $\alpha$ is algebraic and $f,g \in \Z[x]$, the preceding is equivalent to asserting that $f = m_{\alpha} F$
and $g = m_{\alpha} G$, in which $m_{\alpha} \in \Q[x]$ denotes the minimal polynomial of $\alpha$ over $\Q$ and $F,G \in \Q[x]$.
The remarks from the preceding paragraph imply that
$R(f,g) = 0$ if and only if $f,g$ have a common divisor in $\Z[x]$ of positive degree.
\end{proof}

For $f,g \in \C[x]$ with leading coefficients $a_n$ and $b_m$, respectively,
\begin{equation}\label{eq:ResultantProduct}
R(f,g) = a_n^m b_m^n 
\prod_{ \substack{1 \leq i \leq m \\ 1 \leq j \leq n}} (\lambda_i - \mu_j),
\end{equation}
in which $\lambda_1,\lambda_2,\ldots,\lambda_n$ and $\mu_1,\mu_2,\ldots,\mu_m$ are the roots of 
$f$ and $g$ in $\C$, counted by multiplicity \cite[Ex.~31c, p.~621]{DF}.
For $c \in \C$ and $h \in \C[x]$,  \eqref{eq:ResultantProduct} implies
\begin{equation}\label{eq:ResultantProperties}
R(fg,h) = R(f,h) R(g,h)
\qquad \text{and}\qquad
R(f(cx),g(cx)) = c^{mn} R(f,g).
\end{equation}

The \emph{discriminant}
\begin{equation*}
D(f;x) = \frac{(-1)^{n(n-1)/2}}{a_n} R(f(x),f'(x);x)
\end{equation*}
of $f \in \C[x]$ with degree $n$ and leading coefficient $a_n$
is a homogeneous polynomial of degree $2n-2$ in the coefficients of $f$.
The choice of sign ensures that $D(f;x) \geq 0$ if the roots of $f$ are real.  
We often write $D_f$ or $D_{f(x)}$ instead of $D(f;x)$.

If $\deg f = 1$, then $D_f = 1$.
If $f \in \Z[x]$ and $\deg f \geq 2$, then $D_f \in \Z$.  Lemma \ref{Lemma:ZeroResult} ensures that
$D_f = 0$ if and only if $f$ and $f'$ share common divisor in $\Z[x]$ of positive degree; that is, if and only if $f$ has a repeated root in $\C$.
This yields the following.

\begin{lemma}\label{Lemma:IrreducibleDiscriminant}
If $f \in \Z[x]$ is irreducible and has positive degree, then $D_f \neq 0$.
\end{lemma}

The definition and \eqref{eq:ResultantProduct} yield
\begin{equation*}
D_f = a_n^{2n-2} \prod_{1 \leq i < j\leq n} (\lambda_i - \lambda_j)^2,
\end{equation*}
in which $\lambda_1,\lambda_2,\ldots,\lambda_n$ are the roots of $f$, repeated according to multiplicity.
For $\alpha \in \C$, \eqref{eq:ResultantProperties} implies
\begin{equation}\label{eq:DiscriminantFormulas}
D_{f(\alpha x)} = \alpha^{n(n-1)} D_{f(x)}
\qquad \text{and} \qquad
D_{gh} = D_g D_h R_{g,h}^2.
\end{equation}
If $f,g \in \Z[x]$ and $g|f$, then $D_g | D_f$.  If, in addition, $D_f \neq 0$,
then $|D_g| \leq |D_f|$.

\begin{lemma}\label{Lemma:ResultantDivide}
If $f =f_1f_2,\cdots f_k \in \Z[x]$ is a product of distinct irreducible polynomials
$f_1,f_2,\ldots,f_k \in \Z[x]$ of positive degree, then $R(f_i,f_j) | D_f$ for $i\neq j$.
\end{lemma}

\begin{proof}
Since $f_1,f_2,\ldots,f_k \in \Z[x]$ are distinct, irreducible, and have positive degree,
Lemma \ref{Lemma:ZeroResult} ensures that $R(f_i,f_j) \neq 0$ for $i\neq j$.
Moreover, $D_f \neq 0$ since $f$ has no repeated factors.
Induction and \eqref{eq:DiscriminantFormulas} imply $R(f_i,f_j) | D_f$ for $i\neq j$.
\end{proof}

The final lemma of this subsection concerns the number of solutions $\omega_f(p)$ to
$f(x) \equiv 0 \pmod{p}$ for a rational prime $p$.  Note that if $f\in \Z[x]$ is reducible, then
$\omega_f(p) > \deg f$ is possible. For example, consider $3x-6$ modulo $3$.

\begin{lemma}\label{Lemma:Omega}
If $f \in \Z[x]$ is irreducible and $\deg f \geq 1$, then $\omega_f(p) \leq \deg f$ for all $p$.
\end{lemma}

\begin{proof}
If $\deg f \geq p$, then $\omega_f(p) \leq p \leq \deg f$.
Now suppose that $\deg f < p$.  Since $f$ is irreducible, at least 
one coefficient of $f$ is nonzero modulo $p$, and hence $f$ has positive degree modulo $p$.
Lagrange's theorem ensures that $\omega_f(p) \leq \deg f$. 
\end{proof}

\subsection{Dedekind zeta residues}\label{Subsection:Dedekind}
Let $\K$ be a number field of degree $n_{\K}$ with ring of algebraic integers $\O_{\K}$, and
let $\Delta_{\K} $ denote the discriminant of $\K$. 
In what follows, $\PP\subset\O_{\K}$ denotes a prime ideal, $\a\subset\O_{\K}$ an ideal (not necessarily prime), 
and $N(\a)$ the norm of $\a$.
The Dedekind zeta function $\zeta_{\K}(s) = \sum_{\a}N(\a)^{-s}$ of $\K$ is analytic on $\Re{s} > 1$ 
and extends meromorphically to $\C$, except for a simple pole at $s=1$. The analytic class number formula asserts that the residue of $\zeta_{\K}(s)$ at $s=1$ is
\begin{equation*}
\kappa_{\K} = \frac{2^{r_1}(2\pi)^{r_2}h_{\K}R_{\K}}{w_{\K}\sqrt{| \Delta_{\K} |}},
\end{equation*}
in which $r_1$ is the number of real places of $\K$, $r_2$ is the number of complex places of $\K$, $w_{\K}$ is the number of roots of unity in $\K$, $h_{\K}$ is the class number of $\K$, and $R_{\K}$ is the regulator of $\K$ \cite{Lang}.
The nontrivial zeros of $\zeta_{\K}(s)$ lie in the critical strip $0 < \Re s < 1$, in which there might also exist an exceptional, real zero $0 < \beta < 1$. The Generalized Riemann Hypothesis (GRH), which remains unproven, asserts that the nontrivial zeros of $\zeta_{\K}(s)$ satisfy $\Re s = \frac{1}{2}$ and that no exceptional zero exists.

We require unconditional bounds for the residue $\kappa_{\K}$ of $\zeta_{\K}(s)$ at $s=1$.   
If $d=1$, then $\K = \Q$ and hence $\zeta_{\Q}(s)$ is the Riemann zeta function, for which $\kappa_{\Q} = 1$.  Consequently, we restrict our attention to the case $d \geq 2$.

\begin{lemma}\label{lem:residue_bounds}
Let $g \in \Z[x]$ be irreducible with degree $d \geq 2$ and leading coefficient $c$.
If $g(\alpha) = 0$ and $\K = \Q(\alpha)$, then 
\begin{equation}\label{eq:KappaInequality}
     \frac{0.36232}{\sqrt{\D_g}}
    \,\leq\, \kappa_{\K} \,\leq\,
    \left(\frac{e\log \D_g}{2(d-1)}  \right)^{d - 1},
\end{equation}
in which $\D_g = |c|^{(d-1)(d-2)}|D_g|$.
\end{lemma}

\begin{proof}
Let $g \in \Z[x]$ be irreducible with degree $d \geq 2$ and leading coefficient $c$. Then 
$c^{d - 1} g(x) = h(cx)$, in which $h \in \Z[x]$ is monic, irreducible, and has degree $d$.
Let $\alpha \in \C$ be such that $g(\alpha) = 0$.  Then $h(c\alpha) = 0$ and
$\K = \Q(\alpha) = \Q(c\alpha)$; the degree $n_{\K}$ of the number field $\K$ is $d$. 
Since the discriminant of $c^{d - 1} g(x)$ is $(c^{(d-1)})^{2d-2} D_g$ and the 
discriminant of $h(cx)$ is $c^{d(d-1)} D_h$, it follows that
\begin{equation}\label{eq:DiscriminantGH}
D_h = c^{(d-1)(d-2)} D_g.
\end{equation}
From \cite[Prop.~I.2.12]{Neukirch} observe that
\begin{equation*}
    1 \leq \big[\O_{\K} : \Z[c\alpha]\big]^2
    = \bigg( \frac{ \det \Z[c\alpha] }{ \det \O_{\K} } \bigg)^2
    = \frac{|D_h| }{ | \Delta_{\K} |} 
\end{equation*}
and hence
\begin{equation}\label{eq:DiscriminantInequality}
| \Delta_{\K} | \leq |D_h| = \D_g.
\end{equation}
Since $n_{\K} = d \geq 2$,
Louboutin \cite[Thm.~1]{Louboutin00} provides the first inequality in
\begin{equation}\label{eq:KappaUpper}
\kappa_{\K}
    \leq \left(\frac{e\log{| \Delta_{\K} |}}{2(n_{\K} - 1)}\right)^{n_{\K} - 1} 
    \leq \left(\frac{e \log \D_g}{2(d-1)} \right)^{d - 1}.
\end{equation}
The lower bound follows from \cite{EMT4NF2}; the argument is short and reproduced here.
Since $n_{\K} = r_1 + 2r_2\geq 2$, we have $2^{r_1}(2\pi)^{r_2} \geq 2^{2}(2\pi)^{0} = 4$.
Friedman \cite[Thm.~B]{Friedman89} proved that $R_{\K}/w_{\K} \geq 0.09058$, improving on Zimmert \cite{Zimmert} (see \cite[Thm.~7, p.~273]{Lang}).
Thus,
\begin{equation}\label{eq:KappaLower}
    \kappa_{\K} \geq \frac{2^{r_1}(2\pi)^{r_2}R_{\K}}{w_{\K}\sqrt{ | \Delta_{\K}|}} > \frac{4\cdot 0.09058}{\sqrt{\Delta_{\K}}} = \frac{0.36232}{\sqrt{|\Delta_{\K}|}} \geq  \frac{0.36232}{\sqrt{\D_g}}.
\end{equation}
This concludes the proof.
\end{proof}

\begin{remark}\label{Remark:AlternateBound}
With respect to $\D_g$, there are asymptotically superior lower bounds
available (with explicit constants).  For example, one can obtain
\begin{equation*}
\kappa_{\K} > \frac{0.015744605}{d\,d! \,\D_g^{1/d}}
\end{equation*}
from Stark \cite{Stark} by tracking the constants involved \cite{EMT4NF2}.  
However, the dependence upon $d = n_{\K}$ is poor, so
we use the degree-independent lower bound in \eqref{eq:KappaInequality} instead.
\end{remark}

\subsection{The function $\M_{\K}(x)$}\label{Subsection:Mertens}
For a number field $\K$, define the Mertens-type sum
\begin{equation*}
\M_{\K}(x) = \sum_{ N(\PP)\leq x} \frac{1}{N(\PP)},
\end{equation*}
where the sum runs over the prime ideals $\PP$ of $\mathcal{O}_\K$.
If $\K = \Q$, then the sum reduces to the classical Mertens sum 
\begin{equation}\label{eq:ClassicalMertens}
\M_{\Q}(x) = \sum_{p\leq x} \frac{1}{p} \leq \log \log x + C_{\Q} +  \dfrac{1}{(\log x)^2} ,
\end{equation}
in which $C_{\Q} = 0.26149\ldots$ is the Meissel--Mertens constant
and \eqref{eq:ClassicalMertens} holds unconditionally for $x> 1$ \cite[(3.20)]{Rosser}.
For $\K \neq \Q$, the first two authors \cite{EMT4NF2} used an ideal-counting estimate of Sunley 
\cite{SunleyThesis, SunleyBAMS} to prove
unconditionally for $x \geq 2$ that
\begin{equation}\label{eq:MLLMB}
\M_{\K}(x) = \log \log x + C_{\K} + B_{\K}(x),
\end{equation}
in which
\begin{equation}\label{eq:CInequality}
\gamma + \log \kappa_{\K} - d_{\K} \,\leq\, C_{\K} \,\leq\, \gamma + \log \kappa_{\K}
\end{equation}
and
\begin{equation}
|B_{\K}(x)| \leq \frac{2 \Upsilon_{\K}}{\log x}, \label{eq:B}
\end{equation}
where
\begin{equation}\label{eq:Upsilon}
\Upsilon_{\K}=  \left(  \frac{(n_{\K}+1)^2}{2\kappa_{\K}(n_{\K}-1)}\Lambda_{\K}   + 1\right)
+  \frac{0.55\, \Lambda_{\K}n_{\K}(n_{\K}+1)}{\kappa_{\K}} + n_{\K} + 
 40.31 \dfrac{\Lambda_{\K}n_{\K}}{\kappa_{\K}} 
\end{equation}
and
\begin{equation}\label{eq:Lambda}
\Lambda_{\K} 
= e^{28.2n_{\K}+5} (n_{\K}+1)^{ \frac{5(n_{\K}+1)}{2}} | \Delta_{\K} |^{\frac{1}{n_{\K}+1}} 
(\log | \Delta_{\K} |)^{n_{\K}}.
\end{equation}
The Dedekind zeta residue $\kappa_{\K}$ in \eqref{eq:Upsilon} can be
explicitly bounded by Lemma \ref{lem:residue_bounds}.

\subsection{The function $\P(x)$}\label{Subsection:PrimeZeta}

For $x\geq 1$, define
\begin{equation*}
\P(x) = \sum_{2\leq k\leq x} \frac{P(k)}{k},
\end{equation*}
in which $P(s) = \sum_p p^{-s}$ is the prime zeta function. 
Then $\P(1) = 0$ and $\P(n)$ tends monotonically to  $0.31571845\ldots$ as $n \to \infty$; see Table \ref{Table:P}.
That $\lim_{x\to\infty}\P(x)$ exists follows from an old result of Euler (1781) \cite[(2.12), p.~133]{MR958860}:
\begin{equation}\label{eq:EulerBound_for_Px}
    \P(x) < \sum_{k \geq 2} \frac{1}{k} \sum_p \frac{1}{p^k}
    <  \sum_{k = 2}^{\infty} \frac{1}{k}\sum_{n= 2}^{\infty} \frac{1}{n^k} 
    = \sum_{k = 2}^{\infty} \frac{\zeta(k)-1}{k} 
    = 1-\gamma.
\end{equation}%
A convenient upper bound is $\P(x) \leq 0.32$ for all $x \geq 2$.

\begin{table}\small
\begin{equation*}
\begin{array}{cc|cc}
x & \P(x) & x & \P(x)\\
\hline
 1 & 0 & 6 & 0.3136222260 \\
 2 & 0.2261237100 & 7 & 0.3148056307 \\
 3 & 0.2843779231 & 8 & 0.3153133064 \\
 4 & 0.3036262081 & 9 & 0.3155360250 \\
 5 & 0.3107772116 & 10 & 0.3156353854 \\
\end{array}
\end{equation*}
\caption{Approximate values of $\P(x)$.}
\label{Table:P}
\end{table}


\section{Proof of Theorem \ref{Theorem:NFM}}\label{Section:Mertens}

The proof of Theorem \ref{Theorem:NFM} requires the next lemma, which handles the monic case.

\begin{lemma}\label{Lemma:Components1}
Let $h(x) \in \Z[x]$ be monic and irreducible with degree $d \geq 1$.
Let $\alpha \in \C$ be such that $h(\alpha)=0$ and let $\K = \Q(\alpha)$.  
For $x > \sqrt{|D_h|}$,
\begin{equation*}
    \bigg| \M_{\K}(x) - \sum_{p\leq x} \frac{\omega_h(p)}{p} \bigg|
    <  \big( \M_{\Q}(\sqrt{|D_h|}) + 0.64\big) d.
\end{equation*}
\end{lemma}

\begin{proof}
Suppose $h(x) \in \Z[x]$ is monic and irreducible with degree $d \geq 1$.
Let $\alpha \in \C$ be such that $h(\alpha) = 0$ (that is, $h$ is the minimal polynomial of $\alpha$ over $\Q$) 
and consider the number field
$\K = \Q(\alpha)$, which has degree $d$ \cite[p.~47]{StewartTall}. 
Let $\O_{\K}$ denote the ring of algebraic integers in $\K$; by construction it contains $\alpha$.
For each rational prime $p$, the ideal $p\O_{\K}$ factors uniquely (up to reordering the factors) as
\begin{equation*}
p\O_{\K} = \PP_1^{e_1} \PP_2^{e_2}\cdots \PP_r^{e_r},
\end{equation*}
in which $\PP_1,\PP_2,\ldots,\PP_r \subset \O_{\K}$ are distinct prime ideals \cite[Thm.~5.6]{StewartTall}.
Then $N(\PP_i) = \left| \O_{\K}/\PP_i \right| = p^{f_i}$, in which $f_i$ is the inertia degree of $p$ at $\PP_i$, and 
\begin{equation*}
e_1 f_1 + e_2 f_2 + \cdots + e_r f_r = d;
\end{equation*}
see \cite[Prop.~8.2, p.~46]{Neukirch}.
Since each prime ideal $\PP\subset\O_{\K}$ can occur in the factorization for only one rational prime $p$
and, moreover, $p \leq N(\PP) \leq p^d$ \cite[Thm.~5.14c]{StewartTall}, 
the number of prime ideals of norm $p^k$ is at most $d/k$.

The Dedekind factorization criterion relates 
the factorization of $p\O_{\K}$ to the factorization of $h$ over $\F_p$ \cite[Prop.~25, p.~27]{Lang}.
If $p \nmid [\O_{\K} : \Z[\alpha]]$, then 
\begin{equation*}
h(x) = h_1(x)^{e_1}  h_2(x)^{e_2} \cdots h_r(x)^{e_r} 
\end{equation*}
in $\F_p[x]$, where $h_1,h_2,\ldots,h_r \in \F_p[x]$ are distinct and irreducible
with $\deg h_i = f_i$.  
If $a \in \F_p$ and $h(a) = 0$, then there is a unique index $i$ such that $h_i(x) = x-a$.
Hence $f_i = \deg h_i = 1$ and $N(\PP_i) = p$.  Conversely,
$N(\PP_i) = p$ implies $\deg h_i =f_i = 1$ and hence $h_i(x) = x-a$ for some unique $a \in \F_p$.
We conclude that $\omega_h(p)$
equals the number of prime ideals of norm $p$ in the prime ideal factorization of $p\O_{\K}$.

If $p > \sqrt{|D_h|}$, then $p \nmid [\O_{\K} : \Z[\alpha]]$ since
\begin{equation*}
    \big[\O_{\K} : \Z[\alpha]\big]
    = \frac{ \det \Z[\alpha] }{ \det \O_{\K} }
    =\sqrt{ \frac{|D_h| }{ | \Delta_{\K} |}} \leq \sqrt{|D_h|}
\end{equation*}
by \cite[Sec.~1.2]{Neukirch}.
Consequently, $x \geq \sqrt{|D_h|}$ implies
\begin{align*}
   \sum_{p\leq x} \frac{\omega_h(p)}{p}
    &=\sum_{p\leq \sqrt{|D_h|}} \frac{\omega_h(p)}{p} + \sum_{\sqrt{|D_h|} < p\leq x} \frac{\omega_h(p)}{p} \\
    &=\sum_{p\leq \sqrt{|D_h|}} \frac{\omega_h(p)}{p} + \sum_{\substack{\sqrt{|D_h|} < p \leq x\\N(\PP)=p}} \frac{1}{N(\PP)}\\
    &=\M_{\K}(x) + \underbrace{\sum_{p\leq \sqrt{|D_h|}} \frac{\omega_h(p)}{p}  - \sum_{N(\PP) \leq \sqrt{|D_h|}} \frac{1}{N(\PP)}}_{\bigstar} - \underbrace{\sum_{\substack{\sqrt{|D_h|} < p^\ell \leq x\\N(\PP) = p^\ell\\ 2\leq \ell \leq d}} \frac{1}{N(\PP)}}_{\clubsuit}.
\end{align*}
First, note that
\begin{equation}\label{eq:Eh}
    |\bigstar | \leq \max \Bigg\{   \sum_{N(\PP) \leq \sqrt{|D_h|}} \frac{1}{N(\PP)}, \,\,
    \sum_{p\leq \sqrt{|D_h|}} \frac{\omega_h(p)}{p} \Bigg\}
\end{equation}
because the two sums in \eqref{eq:Eh} are both nonnegative.
Since $\P(y) \leq 0.32$ for $y \geq 2$, 
\begin{align*}
    \sum_{N(\PP) \leq \sqrt{|D_h|}} \frac{1}{N(\PP)} 
    \leq \sum_{p^\ell \leq \sqrt{|D_h|}} \frac{d/\ell}{p^\ell} 
    &= d\sum_{p \leq \sqrt{|D_h|}} \frac{1}{p} \,\,+\,\, d\!\!\!\! \sum_{\substack{p^\ell \leq \sqrt{|D_h|}\\2 \leq \ell \leq d}} \frac{1}{\ell\,p^\ell} \\
    &\leq d\,\M_{\Q}(\sqrt{|D_h|}) + d\, \P(\sqrt{|D_h|}) \\
    &\leq d \big(\, \M_{\Q}(\sqrt{|D_h|}) + 0.32\big).
\end{align*}
Lemma \ref{Lemma:Omega} ensures that $\omega_h(p) \leq d$ for all $p$. Thus,
\begin{equation*}
    \sum_{p\leq \sqrt{|D_h|}} \frac{\omega_h(p)}{p} 
    \leq d \sum_{p\leq \sqrt{|D_h|}}\frac{1}{p} 
    = d\, \M_{\Q}(\sqrt{|D_h| }) .
\end{equation*}
To bound \eqref{eq:Eh}, we take the greater of the two estimates above. 

Since the number of prime ideals of norm $p^\ell$ is at most $d/\ell$,
\begin{align*}
    |\clubsuit |
    =
    \sum_{\substack{\sqrt{|D_h|} < p^\ell \leq x\\N(\PP) = p^\ell\\ 2\leq \ell \leq d}} \frac{1}{N(\PP)}
    \leq d \sum_{\ell = 2}^{d} \frac{1}{\ell}\sum_{p\leq x} \frac{1}{p^\ell}
    = d\,\P(x)
    \leq 0.32\,d.
\end{align*}
Combining these observations and using the triangle inequality yields the result.
\end{proof}


We are now ready for the proof of Theorem \ref{Theorem:NFM}.

\begin{proof}[Proof of Theorem \ref{Theorem:NFM}]
Let $g \in \Z[x]$ be irreducible with degree $d \geq 1$ and leading coefficient $c$. Then 
\begin{equation}\label{eq:gh}
    c^{d - 1} g(x) = h(cx),
\end{equation}
in which $h \in \Z[x]$ is monic, irreducible, and has degree $d$.
Let $\alpha \in \C$ be such that $g(\alpha) = 0$.  Then $h(c\alpha) = 0$ and
$\K = \Q(\alpha) = \Q(c\alpha)$.

For $p \nmid c$, \eqref{eq:gh} gives $\omega_g(p) = \omega_h(p)$.
For $p \mid c$, Lemma \ref{Lemma:Omega} yields $\omega_g(p), \omega_h(p) \leq d$.  Thus,
\begin{align}
    \bigg|\sum_{p\leq x} \frac{\omega_g(p)}{p} - \sum_{p\leq x} \frac{\omega_h(p)}{p}\bigg|
    \leq \sum_{p\leq x} \frac{| \omega_g(p) - \omega_h(p)|}{p} 
    &\leq \sum_{p\mid c} \frac{\max\{ \omega_g(p) , \omega_h(p) \}}{p} \nonumber \\
    &\leq d\sum_{p\mid c} \frac{1}{p} \label{eq:Meryn}\\
    &\leq d\, \M_{\Q}(|c|) . \nonumber
\end{align}
Since the discriminant of $c^{d - 1} g(x)$ is $(c^{(d-1)})^{2d-2} D_g$ and the discriminant of $h(cx)$ is $c^{d(d-1)} D_h$, it follows from \eqref{eq:gh} that
\begin{equation}\label{eq:DiscriminantGH}
|D_h| = |c|^{(d-1)(d-2)} |D_g| = \D_g.
\end{equation}
For $x \geq \sqrt{|D_h|}$, Lemma \ref{Lemma:Components1} provides
\begin{align*}
    \left| \M_{\K}(x)  - \sum_{p\leq x} \frac{\omega_g(p)}{p}   \right|
    &= \left| \M_{\K}(x) - \sum_{p\leq x} \frac{\omega_h(p)}{p} \right|
    +\left| \sum_{p\leq x} \frac{\omega_h(p)}{p} - \sum_{p\leq x} \frac{\omega_g(p)}{p} \right| \\
    &< d\big( \M_{\Q}(|c|) + \M_{\Q}(\sqrt{|D_h|}) + 0.64 \big).
\end{align*}
To complete the proof, use \eqref{eq:DiscriminantGH} and rewrite the preceding in terms of $\D_g$. 
\end{proof}

\begin{remark}\label{Remark:Saving}
The term $\M_{\Q}(|c|)$ in \eqref{eq:Ah} can be improved.  Return to \eqref{eq:Meryn} and observe that
the maximal order of $\sum_{p|c}1/p$ is $\log \log \log |c|$ since if
$c=\prod_{p\leq y} p$, then $c=e^{(1+o(1))y}$ and hence $y=(1+o(1))\log c$.  
Consequently,
$\sum_{p|c} 1/p=\sum_{p\leq y} 1/p \sim \log \log y \sim \log\log\log c$.  This can even be made explicit with \cite[Cor.~2.1]{Broadbent}:
\begin{equation*}
\log c = \sum_{p \leq y} \log p = \theta(y) \leq (1 + 1.93378 \times 10^{-8})y.
\end{equation*}
On the other hand,
$\M_{\Q}(|c|) \sim \log \log |c|$ so the dependence upon $c$ is weak already and the improvement
is not worth pursuing.  
\end{remark}

\section{Proof of Theorem \ref{Theorem:Main}}\label{Section:Main}

The proof of Theorem \ref{Theorem:Main} requires two additional lemmas (Subsection \ref{Subsection:Lemmas}).
After that, we estimate three quantities which arise: $\AA$ in Section \ref{Subsection:A}, $\BB(x)$ in Section \ref{Subsection:B}, and $\CC$ in Section \ref{Subsection:C}. Finally, 
we wrap things up in Section \ref{Subsection:Wrap}.

\subsection{Preliminary lemmas}\label{Subsection:Lemmas}
The next two lemmas set up the final estimates needed for the proof of Theorem \ref{Theorem:Main}. Recall that $\omega_f(p)$ denotes the number of solutions to $f(x) \equiv 0 \pmod{p}$. The first lemma concerns the additive structure of $\omega_f(p)$.

\begin{lemma}\label{Lemma:OmegaSplit}
Suppose that $f =f_1f_2,\cdots f_k \in \Z[x]$ is a product of distinct, irreducible nonconstant polynomials $f_1,f_2,\ldots,f_k \in \Z[x]$.
If $p > |D_f|$, then
    \begin{equation*}
        \omega_f(p) = \omega_{f_1}(p) + \cdots + \omega_{f_k}(p).
    \end{equation*}
\end{lemma}

\begin{proof}
Each zero of $f$ in $\Z/p\Z$ is a zero of some $f_i$.  Thus,
\begin{equation}\label{eq:ww}
\omega_f(p) \leq \omega_{f_1}(p) + \cdots + \omega_{f_k}(p)
\end{equation}
holds for any $p$.
Since every zero of $f_i$ in $\Z/p\Z$ is a zero of $f$, it suffices to show that
$f_i$ and $f_j$ have no common zeros in $\Z/p\Z$ if $i \neq j$ and $p > |D_f|$.
If $i \neq j$, then there are $u_{ij}, v_{ij} \in \Z[x]$ such that 
\begin{equation}\label{eq:Bezout}
u_{ij} f_i + v_{ij} f_j = R(f_i,f_j) ,
\end{equation}
in which $R(f_i,f_j) \in \Z$ is the resultant of $f_i$ and $f_j$ \cite[Prop.~5, Sec.~3.6]{Cox}.
Since $f_i,f_j \in \Z[x]$ are distinct and irreducible, they share no common factors over $\C$,
so $R(f_i,f_j) \neq 0$ \cite[Cor.~4, Sec.~3.6]{Cox}.
If $f_i,f_j$ have a common zero in $\Z/p\Z$, then \eqref{eq:Bezout} ensures that
$p \mid R(f_i,f_j)$.  Consequently, $p \mid D_f$ (Lemma \ref{Lemma:ResultantDivide}), 
and hence $p \leq |D_f|$ since $D_f \neq 0$ (Lemma \ref{Lemma:IrreducibleDiscriminant}). 
Thus, $p > |D_f|$ implies \eqref{eq:ww} is an equality.
\end{proof}

The second lemma is an asymptotic estimate for a special sum over $\omega_f(p)$.

\begin{lemma}\label{lem:ADFA_1}
Suppose that $f =f_1f_2,\cdots f_k \in \Z[x]$ is a product of distinct, irreducible, nonconstant polynomials $f_1,f_2,\ldots,f_k \in \Z[x]$
of degrees $d_1,d_2,\ldots,d_k \geq 1$, respectively.
For $x \geq \sqrt{ \D_f}$,
\begin{equation}\label{eq:ABC}
\bigg|\sum_{p \leq x} \frac{\omega_f(p)}{p}  - k \log \log x \bigg|
\leq d\, \M_{\Q}( | D_f| ) + \AA + \BB(x) + \CC,
\end{equation}
in which 
\begin{equation}\label{eq:ddkdeg}
d = d_1+d_2+\cdots+d_k,
\end{equation}
and
\begin{equation*}
\AA = \sum_{i=1}^k |A_{f_i}| ,\qquad
\BB(x) = \sum_{i=1}^k |B_{\K_i}(x)| ,\qquad \text{and}\qquad
\CC = \sum_{i=1}^k |C_{\K_i}|.
\end{equation*}
Here $A_{f_i}$, $B_{\K_i}(x)$, and $C_{\K_i}$ are given by 
\eqref{eq:Ah}, \eqref{eq:B}, and \eqref{eq:MLLMB}, respectively.
\end{lemma}

\begin{proof}
For each $i=1,2,\ldots,k$, let $\alpha_i$ be a root of $f_i$ and define $\K_i =\Q(\alpha_i)$,
which has degree $d_i$ since each $f_i$ is irreducible.  
Observe that $D_f \neq 0$ (since $f$ has no repeated roots in $\C$),
\begin{equation*}
    k \leq d
    \qquad\text{and}\qquad
    1\leq |D_{f_1}|, |D_{f_2}|,\ldots, |D_{f_k}| \leq |D_f|.
\end{equation*}
Without loss of generality, let $c\geq 1$ denote the leading coefficient of $f$ and let 
$c_1,c_2,\ldots,c_k \geq 1$ denote the leading coefficients of $f_1,f_2,\ldots,f_k$, respectively.
Then $c = c_1 c_2 \cdots c_k$ and $c_1,c_2,\ldots,c_k \leq c$.

If $x > |D_f|$, then Lemma \ref{Lemma:OmegaSplit} implies
\begin{align*}
    \sum_{p \leq x} \frac{\omega_f(p)}{p} 
    &= \sum_{p \leq |D_f|}  \frac{\omega_f(p)}{p} +  \sum_{|D_f| < p \leq x}  \frac{ \omega_f(p)}{p} \\
    &= \sum_{p \leq |D_f|}  \frac{\omega_f(p)}{p} +  \sum_{|D_f| < p \leq x}\sum_{i=1}^k   \frac{ \omega_{f_i}(p)}{p} \\
    &= \sum_{p \leq |D_f|}  \frac{\omega_f(p)}{p}  + \sum_{i=1}^k \left(\sum_{p \leq x}  \frac{ \omega_{f_i}(p)}{p}  - \sum_{p \leq |D_f|} \frac{ \omega_{f_i}(p)}{p}\right)\\
    &= \underbrace{\sum_{p \leq |D_f|}  \frac{\omega_f(p)- \sum_{j=1}^{k}\omega_{f_j}(p)}{p} }_{I} 
    + \underbrace{\sum_{i=1}^k \sum_{p \leq x}  \frac{ \omega_{f_i}(p)}{p} }_{II}.
\end{align*}

\noindent\textbf{Term I.} Observe that
\begin{align*}\qquad
    |I|
    &= \left|\sum_{p \leq |D_f|}  \frac{\omega_f(p)-\big( \omega_{f_1}(p)+\omega_{f_2}(p)
        +\cdots + \omega_{f_k}(p)\big)}{p} \right|  \\
    &\leq \sum_{p \leq |D_f|}  \frac{ \omega_{f_1}(p)+\omega_{f_2}(p)+\cdots + \omega_{f_k}(p)}{p} && \text{by \eqref{eq:ww}}   \\
    &\leq \sum_{p \leq |D_f|}  \frac{d_1+d_2+\cdots + d_k }{p}  && \text{by Lemma \ref{Lemma:Omega}}  \\
    &\leq \sum_{p \leq |D_f|} \frac{d}{p}  && \text{by \eqref{eq:ddkdeg}}\\
    &\leq d\,\M_{\Q}(|D_f|).
\end{align*}

\noindent\textbf{Term II.} 
Since $|c_i| \leq |c|$ and $d_i \leq d$, it follows that
\begin{equation}\label{eq:DDF}
\D_{f_i} \leq \D_f
\end{equation}
for $i=1,2,\ldots,k$.  
For $x \geq  \sqrt{\D_f} \geq  \max_{1\leq i \leq k} \sqrt{\D_{f_i}}$,
Theorem \ref{Theorem:NFM} ensures that
\begin{align*}
II
&=  \sum_{i=1}^k \sum_{p \leq x}  \frac{ \omega_{f_i}(p)}{p} \\
&= \sum_{i=1}^k ( \M_{\K_i}(x) + A_{f_i} ) && \text{by Theorem \ref{Theorem:NFM}}\\
&= k \log \log x + \sum_{i=1}^k  (A_{f_i} +B_{\K_i}(x) + C_{\K_i}  ) && \text{by \eqref{eq:MLLMB}}\\
&= k \log \log x + \AA + \BB(x) + \CC.
\end{align*}

Putting this all together, we obtain \eqref{eq:ABC}.
\end{proof}

To complete the proof of Theorem \ref{Theorem:Main}, we must estimate $\AA$, $\BB(x)$, and $\CC$.


\subsection{Estimating $\AA$}\label{Subsection:A}

Since $\M_{\Q}$ and $\P$ are increasing, 
Theorem \ref{Theorem:NFM} and \eqref{eq:ddkdeg} provide
\begin{align*}
\AA
= \sum_{i=1}^k |A_{f_i}|
&\leq \sum_{i=1}^k d_i\Big(\M_{\Q}(|c_i|) + \M_{\Q}\big(\sqrt{\D_{f_i}}\big) + 0.64\Big)\\
&\leq \bigg(\sum_{i=1}^k d_i\bigg)\Big(\M_{\Q}(|c|) + \M_{\Q}\big(\sqrt{\D_{f}}\big) + 0.64\Big) \\
&\leq d\left( \M_{\Q}(|c|) + \M_{\Q}( \sqrt{\D_f}) + 0.64 \right).
\end{align*}

\subsection{Estimating $\BB(x)$}\label{Subsection:B}
The degree $d_i$ of each irreducible factor $f_i$ of $f$ equals the degree $n_{\K_i}$ of 
the number field $\K_i$ generated by a root of $f_i$; that is, $n_{\K_i} = d_i$ for $i=1,2,\ldots,k$.  
Suppose for now that each $d_i \geq 2$.  Then \eqref{eq:DiscriminantInequality} and \eqref{eq:Lambda} imply
\begin{align*}
\Lambda_{\K_i} 
&= e^{28.2 n_{\K_i}+5}(n_{\K_i}+1)^{\frac{5n_{\K_i}+5}{2}} | \Delta_{\K_i}|^{\frac{1}{n_{\K_i}+1}}(\log |\Delta_{\K_i} | )^{n_{\K_i}} \\
&\leq  e^{28.2 d_i+5}(d_i+1)^{\frac{5d_i+5}{2}}|D_{f_i}|^{\frac{1}{d_i+1}}(\log |D_{f_i}|)^{d_i} \\
&\leq  e^{28.2 d+5}(d+1)^{\frac{5d+5}{2}}|D_f|(\log |D_f|)^d
\end{align*}
since $D_{f_i} |  D_f$ and $D_f \neq 0$ (see Lemma \ref{Lemma:IrreducibleDiscriminant} and the discussion below
\eqref{eq:DiscriminantFormulas}).  Define
\begin{equation*}
\Lambda = e^{28.2 d+5}(d+1)^{\frac{5d+5}{2}}|D_f|(\log  |D_f| )^d,
\end{equation*}
which satisfies $0 \leq \Lambda_{\K_i} \leq \Lambda$ for $i=1,2,\ldots,k$ and vanishes if $D_f=1$.

From \eqref{eq:KappaLower}, \eqref{eq:B}, and \eqref{eq:Upsilon}, 
\begin{align}
|B_{\K_i}|
&\leq \frac{2}{\log x}
 \left[ \left(  \frac{(n_{\K_i}+1)^2}{2\kappa_{\K}(n_{\K_i}-1)}\Lambda_{\K_i}   + 1\right)
+  \frac{0.55\, \Lambda_{\K_i}n_{\K_i}(n_{\K_i}+1)}{\kappa_{\K_i}} + n_{\K_i} + 
 40.31 \dfrac{\Lambda_{\K_i}n_{\K_i}}{\kappa_{\K_i}}  \right]  \nonumber \\
&\leq \frac{2 }{\log x}
 \left[ \frac{\Lambda}{\kappa_{\K_i}} \left(   \frac{(n_{\K_i}+1)^2}{2(n_{\K_i}-1)}  
+  0.55\, n_{\K_i}(n_{\K_i}+1) + 
 40.31 n_{\K_i}  \right) + 1 + n_{\K_i} \right] \nonumber \\
&= \frac{2 }{\log x}
 \left[ \frac{\Lambda \sqrt{\D_f}}{0.36232}\left(   \frac{(d_i+1)^2}{2(d_i-1)}  
+  0.55\, d_i(d_i+1) + 
 40.31 d_i  \right) + 1 + d_i \right] \nonumber \\
&\leq \frac{2 }{\log x}
 \left[ \frac{\Lambda \sqrt{\D_f}}{0.36232}\big( 0.5(d_i+7)
+  0.55\, d_i(d +1) + 
 40.31 d_i  \big) + 1 + d_i \right] \label{eq:Jaqen}
\end{align}
for $d_i \geq 2$ and $x \geq 2$.  In the final inequality, we used the fact that
\begin{equation*}
\frac{(x+1)^2}{2(x-1)} \leq \frac{x+7}{2}
\qquad \text{for $x \geq 2$}.
\end{equation*}
This inequality is equivalent to $2-2(x-1)^{-1} \geq 0$, which holds for $x \geq 2$. 
We use this argument to ensure that the choice $d_i=1$ is permissible below.

If $d_i = 1$, then $\K_i = \Q$ and $D_{f_i} = 1$ (the discriminant of a nonconstant linear polynomial is $1$).
Recall from \eqref{eq:ClassicalMertens} that
$M_{\Q}(x) = \log \log x + C_{\Q} + B_{\Q}(x)$,
in which $C_{\Q} = 0.2614972\ldots$ is the Meissel--Mertens constant
and $|B_{\Q}(x)| \leq 2(\log x)^{-2}$ for $x>1$ \cite[(3.17), (3.20)]{Rosser}.  Since
\begin{equation*}
\frac{2}{(\log x)^2} \leq \frac{4 }{\log x} , \qquad
\text{for $x \geq 2$},
\end{equation*}
it follows that \eqref{eq:Jaqen} holds if $d_i = 1$ (even if $|D_f |=1$, so $\Lambda = 0$).  Thus, 
\eqref{eq:ddkdeg} yields
\begin{align*}
\BB(x)
&=  \sum_{i=1}^k |B_{\K_i}(x)|\\
&\leq \sum_{i=1}^k \frac{2 }{\log x}
 \left[ \frac{\Lambda \sqrt{\D_f}}{0.36232}\big( 0.5(d_i+7)+  0.55\, d_i(d +1) +  40.31 d_i  \big) + 1 + d_i \right] \\
&\leq\frac{2 }{\log x} \left(\sum_{i=1}^k 
 \left[ \frac{\Lambda \sqrt{\D_f}}{0.36232}\big( 0.5d_i+3.5+  0.55\, d_i(d +1) +  40.31 d_i  \big) \right] + k + d \right)\\
&\leq\frac{2 }{\log x} \left( 
 \left[ \frac{\Lambda \sqrt{\D_f}}{0.36232}\big( 0.5d+3.5k+  0.55\, d(d +1) +  40.31 d  \big) \right] + 2 d \right)\\
&=\frac{2 }{\log x} \left( 
 \left[ \frac{\Lambda \sqrt{\D_f}}{0.36232}\big(  0.55 d^2  + 44.86 d \big) \right] + 2 d \right).
\end{align*}

\subsection{Estimating $\CC$}\label{Subsection:C}
Recall from \eqref{eq:CInequality} that
\begin{equation*}
    \gamma -d_i +\log \kappa_{\K_i}  \leq C_{\K_i} \leq \gamma +\log \kappa_{\K_i}
\end{equation*}
for $i=1,2,\ldots,k$. Therefore,
\begin{equation*}
|C_{\K_i}| \leq \gamma + d_i + | \log \kappa_{\K_i} |.
\end{equation*}
For $d_i \geq 2$,
the residue bounds \eqref{eq:KappaUpper}, \eqref{eq:KappaLower}, and the inequality \eqref{eq:DDF}, provide
\begin{equation*}
     \frac{0.36232}{\sqrt{\D_f}}
     \,\leq\, \frac{0.36232}{\sqrt{\D_{f_i}}} \
     \,\leq\, \kappa_{\K_i} 
     \,\leq\, \left(\frac{e \log \D_{f_i}}{2(d_i-1)}\right)^{d_i-1} 
     \,\leq\, \left(\frac{e \log \D_{f}}{2}\right)^{d-1}.
\end{equation*}
If $d_i = 1$, then $\K_i = \Q$, $\kappa_{\K_i} = 1$, $D_{f_i}=1$, and $\D_f \geq \D_{f_i} \geq 1$. Thus, 
\begin{equation}\label{eq:KappaFinal}
    \qquad
    \frac{0.36232}{\sqrt{\D_f}} \, \leq \, \kappa_{\K_i}\,
    \leq\,
    \max\left\{1,\frac{e \log \D_{f}}{2}\right\}^{d-1}\qquad\text{for $d_i \geq 1$},
\end{equation}
where the maximum in \eqref{eq:KappaFinal} is  $1$ for $\D_f =1,2$ only. 
Using \eqref{eq:KappaFinal}, for $i=1,2,\ldots,k$ and $d_i\geq 1$,  we have
\begin{align}
    &| \log \kappa_{\K_i}| 
    \leq \max\! \left\{ \log \frac{\sqrt{\D_f}}{0.36232},\, (d-1)\log\max\left\{1,\frac{e \log \D_{f}}{2}\right\} \right\} \nonumber \\
    &\quad\leq \begin{cases} 
    \max\! \left\{ 1.01523 + \frac{1}{2}\log \D_f,\, 0 \right\} &\text{if $\D_f =1,2$,}\\
    \max\! \left\{ 1.01523 + \frac{1}{2}\log \D_f,\, (d-1)\left(\log{\frac{e}{2}} + \log\log \D_{f}\right) \right\} &\text{if $\D_f \geq 3$,}
    \end{cases}\nonumber\\
    &\quad\leq (d-1)
    \begin{cases} 
    1.01523 + \frac{1}{2}\log \D_f &\text{if $\D_f =1,2$,}\\
    \max\! \big\{ 1.01523 + \frac{1}{2}\log \D_f,\, (0.30686 + \frac{1}{2}\log \D_f) \big\} &\text{if $\D_f \geq 3$,}
    \end{cases}\nonumber\\
    &\quad\leq (d-1)\left(1.02 + \frac{1}{2} \log\D_f\right), \nonumber 
\end{align}
since $\log{\D_f} < \sqrt{\D_f}$ because $\D_f \geq 1$.

Since $k \leq d$, and because of \eqref{eq:ddkdeg}, we conclude
\begin{align*}
    \CC
    = \sum_{i=1}^k |C_{\K_i}| \leq \sum_{i=1}^k (\gamma + d_i + | \log \kappa_{\K_i} | ) 
    &\leq d \gamma + d +  (d-1) \sum_{i=1}^k\left(1.02 + \frac{1}{2}\log \D_f\right) \\
    &\leq d \left(\gamma + 1 +  (d-1) 1.02 + \frac{d-1}{2}\log \D_f\right) \\
    &=d \left(\gamma + 1.02\,d - 0.02 + \frac{d-1}{2}\log \D_f\right).
\end{align*}

\subsection{Wrapping things up}\label{Subsection:Wrap}
Now return to \eqref{eq:ABC} and use the bounds on $\AA$, $\BB(x)$, and $\CC$ obtained in 
Subsections \ref{Subsection:A}, \ref{Subsection:B}, and \ref{Subsection:C},  respectively.
This yields the desired bound \eqref{eq:MainUnconditional} of Theorem \ref{Theorem:Main}
and completes the proof. \qed

\section{Conclusion}\label{Section:Conclusion}

A natural number can be proven composite without factoring it.
In a similar manner, Theorem \ref{Theorem:Main} permits us to certify the number of distinct
irreducible factors of $f \in \Z[x]$:
compute how large $x$ must be so that the error in \eqref{eq:MainUnconditional}
is less than $\frac{1}{2}$.

We have not made a concerted effort to obtain the best possible error bounds.
Our main point is that recent advances on explicit Mertens' theorems for number fields \cite{EMT4NF2} finally permit completely explicit results such as our Theorems \ref{Theorem:NFM} and \ref{Theorem:Main}.
Some improvements are required to make this approach computationally feasible.

\subsection{The Generalized Riemann Hypothesis}
The Generalized Riemann Hypothesis will lead to much stronger error terms in the number-field analogues of Mertens' theorems. 
In particular, Theorem \ref{Theorem:Main} might be brought into the realm of computational feasibility, especially if done in conjunction with the following idea.

\subsection{Alternate for Theorem \ref{Theorem:Main}}
One might use a more rapidly divergent series in place of
$\sum_{N(\PP) \leq x} N(\PP)^{-1} \sim \log \log x$. For $x\geq 2$,
the first two authors  \cite{EMT4NF2} proved
\begin{equation}\label{eq:Mertens1}
\Bigg|\sum_{N(\PP)\leq x} \frac{\log N(\PP)}{N(\PP)} - \log x \Bigg| \leq \Upsilon_{\K},
\end{equation}
with $\Upsilon_{\K}$ as in \eqref{eq:Upsilon}.
One hopes for an effective version of 
Theorem \ref{Theorem:Main} of the form
\begin{equation}\label{eq:Main_adapted}
    k = \frac{1}{\log{x}}\sum_{p\leq x}\frac{\omega_f(p)\log{p}}{p} + O\left(\frac{1}{\log{x}}\right).
\end{equation}
This possibility is suggested by an 
inexplicit result of Nagell \cite{Nagell}, whose error term was improved by Bantle \cite{Bantle}:
\begin{equation}\label{eq:Nagell}
    \sum_{p\leq x}\frac{\omega_f(p)\log{p}}{p} = \log{x} + O(1),
\end{equation}
Diamond--Halberstam provide an alternate proof \cite[Prop.~10.1]{DiamondHalberstam}, and they use this result in their sieve methods. Moreover, Halberstam--Richert use a generalized version
(see equation $(\Omega(\kappa,L))$ on \cite[p.~142]{HalberstamRichert}) of \eqref{eq:Nagell}  throughout their treatise on sieve methods \cite{HalberstamRichert}. Therefore, results like 
Theorem \ref{Theorem:NFM} are of independent interest, for example, in the development of explicit sieve methods.

\bibliographystyle{amsplain} 
\bibliography{NDIFP}

\end{document}